\documentclass[11pt]{amsart}

\usepackage[utf8x]{inputenc} 	      
\usepackage{lmodern}                
\usepackage[T1]{fontenc}            
\usepackage[english]{babel}         
\usepackage{ucs}             	      
\usepackage{hyperref}               
\usepackage{geometry}
\usepackage{graphicx}     
\usepackage{xcolor}       
\usepackage{enumitem}     

\usepackage{amsmath, amsthm}
\usepackage{amsfonts, amssymb}
\usepackage{bm}               

\theoremstyle{plain}
	\newtheorem{theorem}{Theorem}[section]
	\newtheorem*{theorem*}{Theorem}
	\newtheorem{lemma}{Lemma}[section]
	\newtheorem{proposition}{Proposition}[section]
	\newtheorem*{proposition*}{Proposition}

\numberwithin{equation}{section}

\theoremstyle{remark}
	\newtheorem{remark}{\textbf{Remark}}[]
	
	\newtheorem*{example}{Example}

\theoremstyle{definition}

\newcommand{\ve}{\varepsilon}
\newcommand{\eg}{\emph{e.g. }}

\newcommand{\RR}{\mathbb{R}}
\newcommand{\NN}{\mathbb{N}}

\newcommand{\TT}{\mathbb{T}}

\newcommand{\Q}{\mathcal{Q}}

\newcommand{\R}{ {\mathbb R} }

\newcommand{\eps}{{\varepsilon}}
\newcommand{\ind}{{\mathbb I}}

\author[P.-E. Jabin]{Pierre-Emmanuel Jabin}
 \address{{\sc Pierre-Emmanuel Jabin}, CSCAMM and Dept. of Mathematics, University of Maryland,
College Park, MD 20742, USA. P.--E.~{\sc Jabin} is partially supported by NSF Grant 1312142 and by NSF Grant RNMS (Ki-Net) 1107444.}
\email{pjabin@umd.edu}

\author[T. Rey]{Thomas Rey}
\address{{\sc Thomas Rey},
Laboratoire P. Painlevé, CNRS UMR 8524, Université Lille 1,
59655 Villeneuve d'Ascq Cedex, France. RT. {\sc Rey} is partially supported by the team Inria/Rapsodi,  Labex CEMPI (ANR-11-LABX-0007-01) and NSF Grant RNMS (Ki-Net) 1107444.}

\email{thomas.rey@math.univ-lille1.fr}

\title{Hydrodynamic limit of granular gases to pressureless Euler in dimension 1}

\date{}

\begin{document}

  \begin{abstract}
We investigate the behavior of granular gases in the limit of small Knudsen number, that is very frequent collisions.   
          We deal with the strongly inelastic case, in one dimension of space and velocity. We are able to prove the convergence toward the pressureless Euler system. The proof relies on dispersive relations at the kinetic level, which leads to the so-called Oleinik property at the limit.
  \end{abstract}
  \maketitle
  
  \tableofcontents

  \section{Introduction}
   The granular gases equation is a Boltzmann-like kinetic equation describing a rarefied gas composed of macroscopic particles, interacting via energy-dissipative binary collisions (pollen flow in a fluid, or planetary rings for example). More precisely, the phase space distribution $ f^\ve(t,x,v)$ solves the equation
		\begin{equation}
		  \label{eqBoltzEPS}
		  \left\{ \begin{aligned}
			  & \frac{\partial f^\ve}{\partial t} + v \frac{\partial f^\ve }{\partial x}= \frac{1}{\varepsilon} \mathcal Q_\alpha(f^\ve,f^\ve),\\
			  & \, \\
			  & f^\ve(0,x,v) = f_\eps^{0}(x,v),
			\end{aligned} \right.
		\end{equation}
		where  $f^{0}_\eps$ is a given non negative distribution, $t \geq 0$, $v \in \RR$ and $x\in \RR$.
		The collision operator $\Q_\alpha$ is the so-called \emph{granular gases} operator (sometimes known as the inelastic Boltzmann operator), describing an energy-dissipative microscopic collision dynamics, which we will present in the following section.
		The parameter $\ve > 0$ is the scaled \emph{Knudsen} number, that is the ratio between the mean free path of particles before a collision and the length scale of observation.

As $\ve\rightarrow 0$, the frequency of collisions increases to infinity.  The \emph{particle distribution function} $f^\ve$  then formally converges towards a Dirac mass centered on the mean velocity, 
	\begin{equation}
		  \label{defMonokinetic}
		  \rho(x)\delta_0\left (v-{u}(x)\right ), \quad \forall (x,v) \in \RR\times\RR.
		\end{equation}
This is due to the energy dissipation which ensures that all particles occupying the same position in space, necessarily have the same velocity.
 
The form \eqref{defMonokinetic} of $f^\ve$ is usually called monokinetic and greatly reduces the complexity of Eq. \eqref{eqBoltzEPS}: The solution is completely described by its local hydrodynamic fields, namely its \emph{mass} $\rho \geq 0$ and its \emph{velocity} $u \in \RR$.

Before the limit $\ve\rightarrow 0$, the same \emph{macroscopic} quantities can be obtained from the distribution function $f^\ve$ by computing its first moments in velocity:
		\begin{equation}
	    \label{defMacroQuantities}
	    \rho^\ve(t,x) \, = \, \int_{\RR} f^\ve(t,x,v) \, dv, \qquad 
	    \rho^\ve(t,x) \, {u}^\ve(t,x) \, = \, \int_{\RR} f^\ve(t,x,v) \, v \, dv.
	  \end{equation}
However those quantities cannot be solved independently as they do not satisfy a closed system for $\ve>0$, instead one has by integrating Eq. \eqref{eqBoltzEPS} (see the properties of the collision operator just below)
\[
\begin{split}
&\partial_t \rho^\ve+\partial_x(\rho^\ve\,u^\ve)=0,\\
&\partial_t (\rho^\ve\,u^\ve)+\partial_x(E^\ve)=0,
\end{split}
\]
where $E^\ve=\int_\R f^\ve(t,x,v) |v|^2 \,dv$ and cannot be expressed directly in terms of $\rho^\ve$ and $u^\ve$. But at the limit $\ve\rightarrow 0$, if \eqref{defMonokinetic} holds, then one has that $E=\rho\,u^2$ and $\rho,\;u$ now satisfy the pressureless Euler dynamics
\begin{equation}
\left\{\begin{aligned}
&\partial_t \rho+\partial_x(\rho\,u)=0,\\
&\partial_t (\rho\,u)+\partial_x(\rho\,u^2)=0.
\end{aligned}\right.\label{sysSticky}
\end{equation}
This system of equation is mostly known as a model for the formation of large scale structures in the universe (\eg aggregates of galaxies) \cite{SilkSzalayZeldovich:1983}.

The purpose of this article is to justify rigorously this limit of Eq. \eqref{eqBoltzEPS} to \eqref{sysSticky}. 

Such hydrodynamic limits for collisional models have been famously investigated for {\em elastic collisions} (preserving the kinetic energy) such as the Boltzmann equation. They are connected to the rigorous derivation of Fluid Mechanics models (such as incompressible Navier-Stokes or Euler); this longstanding conjecture formulated by Hilbert was finally solved in     \cite{GolseStRaymond2, GolseStRaymond1, StRaymond}.

The inelasticity (loss of kinetic energy for each collision) leads however to a very distinct behavior and requires different techniques. In fact even classical formal techniques such as Hilbert or Chapman-Enskog expansions (see \eg \cite{CIP:94} for a mathematical introduction in the elastic case) are not applicable. The limit system for instance is very singular (see the corresponding subsection below), to the point that well posedness for \eqref{sysSticky} is only known in dimension $1$. This is the main reason why our study is limited to this one-dimensional case.

We continue this introduction by explaining more precisely the collision operator. We then present the current theory for the limit system \eqref{sysSticky} before giving the main result of the article.

  	  

	  \subsection{The Collision Operator}
	
			Let $\alpha \in [0,1]$ be the restitution coefficient of the microscopic collision process, that is the ratio of kinetic energy dissipated during a collision, in the direction of impact. This quantity can depend on the magnitude of the relative velocity before collision $|v-v_*|$ (see the book \cite{brilliantov:2004} for a long discussion of this topic).
			
			If $\alpha = 1$, no energy is dissipated, and the collision is \emph{elastic}. If $\alpha \in (0,1)$, the collision is said to be \emph{inelastic}. 
We define a strong form of the \emph{collision operator} $\mathcal Q_\alpha$ by
			\begin{align}
			  \mathcal Q_\alpha(f, g)(v) & = \int_\mathbb{R} |v-v_*| \left( \frac{\, f' \, g_*'}{\alpha^2} - f \,g_* \right) dv_*, \label{defBoltzOp}\\
			   & = \mathcal Q_\alpha^+(f,g)(v) - f(v) L(g)(v) \notag,
			\end{align}
			where we have used the usual shorthand notation $\, f' := f(v')$, $\, f_*' := f(v_*')$, $f := f(v)$, $f_* := f(v_*)$. 
In \eqref{defBoltzOp}, $\, v'$ and $\, v_*'$ are the pre-collisional velocities of two particles of given velocities $v$ and $v_*$, defined by
			\begin{equation}
			  \label{defMicroDynamicsInel}
				\left\{\begin{aligned} 
					& v' = \frac{1}{2}(v + v_*) + \frac{\alpha}{2} (v - v_*), \\
					& v_*' = \frac12 (v + v_*) - \frac{\alpha}{2} (v - v_*).
				\end{aligned} \right.
			\end{equation}
		  The operator $\mathcal Q_\alpha^+(f,g)(v)$ is usually known as the \emph{gain} term because it can be understood as the number of particles of velocity $v$ created by collisions of particles of pre-collisional velocities $v'$ and $v_*'$, whereas $f(v)L(g)(v)$ is the \emph{loss} term, modeling the loss of particles of pre-collisional velocities $v'$. 
		  
		  We can also give a weak form of the collision operator, which is compatible with sticky collisions. 
		  Let us reparametrize the post-collisional velocities $v'$ and $v_*'$ as
		  \begin{equation*}
				\left \{\begin{aligned} 
					v' & = v - \frac{1 - \alpha}{2} (v - v_*), \\
					v_*' & = v_* + \frac{1 - \alpha}{2} (v - v_*) .
				\end{aligned}\right .
			\end{equation*}
			Then we have the weak representation, for any smooth test function $\psi$,
			\begin{equation} 
			  \label{defBoltzweak}
				\int_{\RR} Q_\alpha(f,g) \, \psi(v) \, dv = \frac{1}{2 }\int_{\mathbb{R} \times \mathbb{R}} |v-v_*|f_{*} \, g \, \left(\psi' + \psi_*' - \psi - \psi_* \right) dv \, dv_*.
			\end{equation}

			Thanks to this expression, we can compute the macroscopic properties of the collision operator $\Q_\alpha$. Indeed, we have the microscopic conservation of impulsion and dissipation of kinetic energy:
			\begin{align*}
				v' + v_*' & = v + v_*, \\
				(v')^2 + (v_*')^2 - v^2 - v_*^2 & = - \frac{1-\alpha^2}{2} ( v-v_* )^2 \leq 0.
			\end{align*}
			Then if we integrate the collision operator against $\varphi(v) = (1, \, v, \, v^2)$, we obtain the preservation of mass and momentum and the dissipation of kinetic energy:
			\begin{equation}
			  \label{eqMacroOp}
		    \int_\RR  \Q_\alpha(f,f)(v) \begin{pmatrix} 1 \\ v \\ v^2 \end{pmatrix} dv \,=\, \begin{pmatrix} 0 \\ 0 \\ - (1-\alpha^2) D(f,f) \end{pmatrix},
		  \end{equation}
		  where $D(f,f) \geq 0$ is the \emph{energy dissipation} functional, given by 
		  \begin{equation}
		  	\label{defDissipFunctional}
		    D(f,f) := \int_{\mathbb{R} \times \mathbb{R}}f \, f_* \, |v-v_*|^3 \, dv \, dv_* \geq 0.
		  \end{equation}
		  The conservation of mass implies an \emph{a priori} bound for $f$ in $L^\infty\left (0,T; \; L^1(\RR \times \RR)\right )$.
		  Moreover, these macroscopic properties of the collision operator, together with the conservation of positiveness, imply that the equilibrium profiles of $\Q_\alpha$ are trivial Dirac masses (see \eg the review paper \cite{Villani:2006} of Villani).
		  

	  
		  Finally, we can give a precise estimate of the energy dissipation functional. Indeed, applying Jensen's inequality to the convex function $v \mapsto |v|^3$ and to the measure $f(v_*)\, dv_*$, we get
			\begin{align*}
				\int_{\mathbb{R}} f(v_*) |v-v_*|^3 dv_* \geq & \left| v \int_\mathbb{R} f(v_*) \, dv_* - \int_\mathbb{R} v_* \, f(v_*) \, dv_* \right| = \left |\rho \, (v - u) \right |^3 .
			\end{align*}
			Using H\"older inequality, it comes that the energy dissipation is such that
			\begin{align} 
			  D(f,f) & \geq \rho^3  \int_\RR f(v) \, |v - u|^3 \, dv \notag \\
			   & \geq \rho^{5/2} \left( \int_\RR f(v) \, |v-u|^2 dv \right)^{\frac{3}{2}}.
			   \label{ineqBoundDissipEner}
			\end{align}	

     \begin{remark}
       Let us define the \emph{temperature} of a particle distribution function $f$ by.
       \[ \theta(t,x) := \int_\RR |v-u|^2 \, f(t,x,v) \, dv.\]
       Multiplying equation \eqref{eqBoltzEPS} by $|v-u|^2$ and integrating with respect to the velocity and space variables yields thanks to \eqref{ineqBoundDissipEner} the so-called \emph{Haff's} Law \cite{haff:1983}
       \begin{equation}
         \label{ineqHaffLaw}
         \int_\RR \theta^\ve(t,x) \, dx \, \lesssim \, \frac{1-\alpha}\ve \frac1{(1+t)^2}.
       \end{equation}
       This asymptotic behavior of the macroscopic temperature is characteristic of granular gases, and has been proved to be optimal in the space homogeneous case for constant restitution coefficient by Mischler and Mouhot in \cite{MischlerMouhot:20062}. 
       These results have then been extended to a more general class of collision kernel and restitution coefficients by Alonso and Lods in \cite{AlonsoLods:2010,AlonsoLods:2013} and by the second author in \cite{rey:2011}.
       Nevertheless, in all these works, additional constraints on the smoothness of the initial data  (a somehow nonphysical $L^p$ bound for $p>1$) are required for the results to hold.
     \end{remark}

The existence in the general $\RR_x^3 \times \RR_v^3$ setting, for a large class of velocity-dependent restitution coefficient but close to vacuum was obtained in \cite{Alonso:2009}. The stability in $L^1(\RR^3_x \times \RR^3_v)$ under the same assumptions was derived for instance in \cite{Wu:2009}. Finally the existence and convergence to equilibrium in $\TT_x^3 \times \RR_v^3$  for a diffusively heated, weakly inhomogeneous granular gas was proved in \cite{Tristani:2013}.

As one can imagine, the theory in the dimension $1$ case (as concerns us here) is much simpler. The existence of solutions for the granular gases equation \eqref{eqBoltzEPS} in one dimension of physical space and velocity, with a constant restitution coefficient was proved in \cite{BenedettoCagliotiPulvirenti:97} for compact initial data. The velocity-dependent restitution coefficient case, for small data, was then proven in \cite{BenedettoPulvirenti:2002}. More precisely, one has:

\begin{theorem}[From \cite{BenedettoPulvirenti:2002}]    
	Let us assume that there exists $\gamma \in (0,1)$ such that
	\[
		\alpha=\alpha(|v-v_*|) = \frac1{1+|v-v_*|^\gamma}.
	\] 
	Then, for $0\leq f^0 \in L^{\infty}(\RR_x \times \RR_v)$ with small total mass, there exists an unique mild, bounded solution in $L^\infty(\RR_x \times \RR_v)$ of \eqref{eqBoltzEPS}.
\end{theorem}
The main argument is reminiscent from a work due to Bony in \cite{bony:1987} concerning discrete velocity approximation of the Boltzmann equation in dimension $1$. 

Finally, the problem of the hydrodynamic limit was only tackled formally, and in the quasi-elastic setting $\alpha \to 1$. The first results for this case can be found in \cite{BenedettoCagliotiGolsePulvirenti:1999} for the one dimensional case. The review paper \cite{Toscani:2008} summarizes most of the known formal results for the general case.

   \subsection{Pressureless Euler: The Sticky Particles dynamics}
%
The pressureless system \eqref{sysSticky} is rather delicate. It can (and will in general) exhibits shocks as the velocity $u$ formally solves the Burgers equation where $\rho>0$. The implied lack of regularity on $u$ leads to concentrations on the density $\rho$ which is only a non-negative measure in general.

System \eqref{sysSticky} is hence in general ill-posed as classical solutions cannot exists for large times and weak solutions are not unique. It is however possible to recover a well posed theory by imposing a semi-Lipschitz condition on $u$. This theory was introduced in \cite{BouchutJames:1999}, and later extended in \cite{boudin:2000} and \cite{HuangWang:2001} (see also \cite{ERykovSinai:1996} and \cite{ChertockKurganovRykov:2007}). We cite below the main result of \cite{HuangWang:2001}, where $M^1(\RR)$ denotes the space of Radon measures on $\RR$ and $L^2(\rho)$ for $\rho\geq 0$ in $M^1(\RR)$ denotes the space of functions which are square integrable against $\rho$.
\begin{theorem} [From \cite{HuangWang:2001}]
For any $\rho^0\geq 0$ in $M^1(\RR)$ and any $u^0\in L^2(\rho^0)$, there exists $\rho\in L^\infty(\RR_+, M^1(\RR))$ and $u\in L^\infty(\RR_+, L^2(\rho))$ solution to System \eqref{sysSticky} in the sense of distribution and satisfying the semi-Lipschitz Oleinik-type bound
\begin{equation}
u(t,x)-u(t,y)\leq \frac{x-y}{t},\quad \mbox{for}\ a.e.\;x> y.\label{Oleinik0}
\end{equation}
Moreover the solution is unique if $u^0$ is semi-Lipschitz or if the kinetic energy is continuous at $t=0$
\[
\int_\RR \rho(t,dx)\,|u(t,x)|^2\longrightarrow \int_\RR \rho^0(dx)\,|u^0(x)|^2,\qquad \mbox{as}\ t\rightarrow 0.
\]\label{wellposedpressureless}
\end{theorem}  
The proof of Th. \ref{wellposedpressureless} is quite delicate, relying on duality solutions. For this reason, we only explain the rational behind the bound \eqref{Oleinik0}, which can be seen very simply from the discrete sticky particles dynamics. We refer in particular to \cite{BrenierGrenier:1998} for the limit of this sticky particles dynamics as $N\rightarrow \infty$.
  
Consider $N$ particles on the real line. We describe the $i^{\text{th}}$ particle at time $t> 0$ by its position $x_i(t)$ and its velocity $v_i(t)$. Since we are dealing with a one dimensional dynamics, we can always assume the particles to be initially ordered  
     \[
x_1^{in} < x_2^{in} < \ldots < x_N^{in}. 
\]
The dynamics is characterized by the following properties     
     \begin{enumerate}[label=(\roman{*})]
 \item The particle $i$ moves with velocity $v_i(t)$: $\frac{d}{dt} x_i(t)=v_i(t)$.     
 \item The velocity of the $i^{\text{th}}$ particle is constant, as long as it does not collide with another particle: $v_i(t)$ is constant as long as $x_i(t)\neq x_j(t)$ for all $j\neq j$.
       \item The velocity jumps when a collision occurs: if at time $t_0$ there exists $j \in \{1, \ldots, N \}$ such that $x_j(t_0 ) = x_i(t_0)$ and $x_j(t) \neq x_i(t)$ for any $t < t_0$, then all the particles with the same position take as new velocity the average of all the velocities
         \begin{equation*}
           \label{defStickyDynamics}
           v_i(t_0+) = \frac{1}{|{j | x_j(t_0 ) = x_i(t_0)}|}\sum_{j | x_j(t_0 ) = x_i(t_0)} v_j(t_0-).
         \end{equation*}
     \end{enumerate}
    Note in particular that particles having the same position at a given time will then move together at the same velocity. Hence, only a finite number of collisions can occur, as the particles aggregates.

This property also leads to the Oleinik regularity. Consider any two particles $i$ and $j$ with $x_i(t)>x_j(t)$. Because they occupy different positions, they have never collided and hence $x_i(s)>x_j(s)$ for any $s\leq t$. If neither had undergone any collision then $x_i(0)=x_i(t)-v_i(t)\,t>x_j(0)=x_j(t)-v_j(t)\,t$ or
\begin{equation}
	     \label{eqMicroBound}
	     \frac{\left (v_{i} - v_j\right )_+}{\left (x_{i}-x_j\right )_+} < \frac{1}{t}, 
	   \end{equation}
     where $x_+ := \max (x,0)$. It is straightforward to check that \eqref{eqMicroBound} still holds if particles $i$ and $j$ had some collisions
between time $0$ and $t$.

As one can see this bound is a purely dispersive estimate based on free transport and the exact equivalent of the traditional Oleinik regularization for Scalar Conservation Laws, see \cite{Oleinik}. It obviously leads to the semi-Lipschitz bound \eqref{Oleinik0} as $N\rightarrow \infty$.
  
  We conclude this subsection with the following remark which foresees our main method. 
  
\begin{remark}
\label{remEqMeasure}
Define the empirical measure of the distribution of particles 
     \begin{equation}
       \label{defEmpiricalMeasure}
       f_N(t,x,v) := \sum_{i=1}^N \delta_0\left (x - x_i(t)\right ) \delta_0\left (v - v_i(t)\right ).
     \end{equation}
The empirical measure is solution to the following kinetic equation
       \begin{equation}
         \partial_t f_N + v \, \partial_x f_N \, = \, - \partial_{vv}m_N,
\label{kineticform}
\end{equation}
       for some non-negative measure $m_N$. This equation embeds the fundamental properties of the dynamics: conservation of mass and momentum, and dissipation of kinetic energy. It is in several respect a sort of kinetic formulation, rather similar to the ones introduced for some conservation laws \cite{LionsPerTad, LionsPerTad2}, see also \cite{Perthame}.

The kinetic formulation \eqref{kineticform} has to be coupled with a constraint on $f_n$ (just like for Scalar Conservation Laws). Unsurprisingly this constraint is that $f_N$ has to  be monokinetic
\[
f_N=\rho_N(t,x)\,\delta(v-u_N(t,x)).
\]
\end{remark}

		\subsection{Main Result}
We are now ready to state the main result of this article.		  
\begin{theorem}\label{mainresult}
Consider a sequence of weak solutions $f_\eps(t,x,v)\in L^\infty([0,\ T],\ L^p(\R^2))$ for some $p>2$ and with total mass $1$ to the granular gases Eq. \eqref{eqBoltzEPS} such that all initial $v$-moments are uniformly bounded in $\eps$
\begin{equation}
\sup_\eps\int_{\R^2} |v|^k\,f^0_\eps(x,v)\,dx\,dv<\infty,\label{initialmoments}
\end{equation}
some moment in $x$ is uniformly bounded, for instance
\begin{equation}
\sup_\eps \int_{\R^2} |x|^2\,f_\eps^0(x,v)\,dv<\infty,
\end{equation}
and $f^0_\eps$ is, uniformly in $\eps$, in some $L^p$ for $p>1$
\begin{equation}
\sup_\eps\int_{\R^2} (f^0_\eps(x,v))^p\,dx\,dv<\infty.\label{initiallp}
\end{equation}
Then any weak-* limit $f$ of $f_\eps$ is monokinetic, $f=\rho(t,x)\,\delta(v-u(t,x))$ for $a.e.\ t$, where $\rho,\;u$ are a solution in the sense of distributions to the pressureless system \eqref{sysSticky} while $u$ has the Oleinik property for any $t>0$
\[
u(t,x)-u(t,y)\leq \frac{x-y}{t},\quad \mbox{for}\ \rho\ a.e.\ x\geq y.
\]
\end{theorem}		
\begin{remark}
It is possible to replace the $L^p$ condition on $f^0$ by assuming that $f^0_\eps$ is well prepared in the sense that $f_\eps^0\rightarrow \rho^0\,\delta(v-u^0(x))$ for some $u^0$ Lipschitz with the convergence in an appropriate sense (made precise in Remark \ref{rkuniqueness} after Theorem \ref{generalhydro}). In that case one knows in addition that the limit is the unique ``sticky particles'' solution to the pressureless system \eqref{sysSticky} as obtained in \cite{BouchutJames:1999,HuangWang:2001}  
\end{remark}
The basic idea of the proof of Th. \ref{mainresult} is to use the kinetic description \eqref{kineticform} to compare the granular gases dynamics to pressureless gas system. The main difficulty is to show that $f_\ve$ becomes monokinetic at the limit. This is intimately connected to the Oleinik property \eqref{Oleinik0}, just as this property is critical to pass to the limit from the discrete sticky particles dynamics.

Unfortunately it is not possible to obtain \eqref{Oleinik0} directly. Contrary to the sticky particles dynamics, this bound cannot hold for any finite $\ve$ (or for any distribution that is not monokinetic). This is the reason why it is very delicate to obtain the pressureless gas system from kinetic equations (no matter how natural it may seem). Indeed we are only aware of one other such example in \cite{KangVasseur:2014}.

One of the main contributions of this article is a complete reworking of the Oleinik estimate, still based on dispersive properties of the free transport operator $v\,\partial_x $ but compatible with kinetic distributions that are not monokinetic.

The next section is devoted to the introduction and properties of the corresponding new functionals. This will allow us to prove a more general version of Th. \ref{mainresult} in the last section.
  \section{A New Dissipative Functional for kinetic equations}
    \label{secFunc}
\subsection{Basic Definitions}
    The heart of our proof relies on new dissipative properties of kinetic equations which are
\begin{itemize}
\item Contracting in velocity;
\item Close to monokinetic.
\end{itemize}
Mathematically speaking, consider $f\in L^\infty([0,\ T], M^1(\R^2)$ solution to
  \begin{equation}
         \label{basickinetic}
         \partial_t f+ v \, \partial_x f \, = \, - \partial_{vv} m,\qquad m\in M^1([0,\ T]\times\R^2),\quad m\geq0. 
       \end{equation}
We also need a notion of trace for $f$ and more precisely that
\begin{equation}
\label{traceonx}
\Lambda_{f,k}(t)=\limsup_{\delta\rightarrow 0}\frac{1}{\delta}\int_{x\in \R}\int_{x<y<x+\delta}\int_{v,w\in \R^2} (v-w)_+^k f(t,x,v)\,f(t,y,w)\,dv \,dw\,dy\,dx\in L^1([0,\ T]).
\end{equation} 
This system is now dissipative and will yield as a dissipation rate a control   on the following nonlinear functional for any $\eta$, $\mu>0$, $k\geq 1$     
    \begin{equation}
      \label{defFunctionalL}
      \mathcal L_{\eta,\mu,k}(f)(t) := \int  \frac{(v-w)_+^{k+2}}{( x-y + \eta)^k}\, \chi_\mu (x-y) \, f(t,x,v) f(t,y,w) \,dv\,dw\,dx\,dy,
    \end{equation}	
    where the function $\chi_\mu$ is a smooth, non-centered approximation of the Heaviside function, as in Figure \ref{figHeavisideSmooth}. In particular $\chi_\mu$ is non-increasing in $\mu$ and
\begin{equation}
0\leq \chi_\mu(x) \leq \ind_{x>0},\quad 0\leq \chi_\mu'(x)\leq \frac{2}{\mu}\,\ind_{0<x<mu}.\label{hypchi}
\end{equation}

    This functional is somehow similar to the one described by Bony in \cite{bony:1987}, and used by Cercignani in \cite{Cercignani:1992} and by Biryuk, Craig and Panferov in \cite{BiCrPa:2006}.

To make notations consistent, we define when $k=0$
\begin{equation}
      \label{defFunctionalLk0}
      \mathcal L_{\eta,\mu,0}(f)(t) := -\int  \left (v-w\right )_+^{2}\, \log( x-y + \eta)_-\, \chi_\mu (x-y) \, f(t,x,v) f(t,y,w) \,dv\,dw\,dx\,dy.
    \end{equation}       
We also define, from the monotonicity of $\chi_\mu$
\begin{equation}
      \label{defFunctionalLmu0}
      \mathcal L_{\eta,0+,k}(f)(t) := \sup_{\mu\rightarrow 0} \mathcal L_{\eta,\mu,k}(f)(t)=\limsup_{\mu\rightarrow 0} \mathcal L_{\eta,\mu,k}(f)(t).
    \end{equation}
Observe that for $\mu=0$, $\mathcal L_{\eta,0,k}(f)(t)$  may not be well defined and may in fact depend on the way the Heavyside function $\mathcal I(x-y)$ is approximated. This is the reason for the precise definition above of $\mathcal L_{\eta,0+,k}(f)(t)$. Furthermore from the trace property \eqref{traceonx}, whatever the definition of $\mathcal L_{\eta,0,k}(f)(t)$, one would have that $\int_0^T \mathcal L_{\eta,0,k}(f)(t)\,dt\leq \int_0^T \mathcal L_{\eta,0+,k}(f)(t)\,dt+2\int_0^T\Lambda_{f,k}(t)\,dt$ as explained below.    
    \begin{figure} 
      \includegraphics[trim = 80 225 80 45, clip,scale=1]{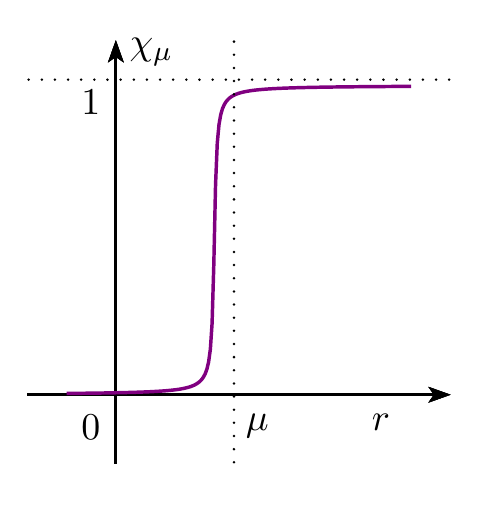}
			\caption{Smoothed, non-centered approximation of the Heaviside function $\chi_\mu$.}
			\label{figHeavisideSmooth}
		\end{figure}

    \begin{example}
It is possible to prove that $\mathcal L_{\eta,\mu,k}$ is bounded  for the sticky particle dynamics. Indeed, let $\left (x_i(t), v_i(t)\right )_{1 \leq i \leq N}$ for $N \in \NN$ be solution to the sticky particles system \eqref{defStickyDynamics} and $f_N$ be the associated empirical measure given by \eqref{defEmpiricalMeasure}. We already observed in Remark \ref{remEqMeasure} that $f_N$ solves \eqref{basickinetic}; moreover it has the trace property \eqref{traceonx} with $\Lambda_{f,k}=0$. 

In that simple example, it is possible to bound $\mathcal L_{\eta,\mu,k}$ directly by using \eqref{eqMicroBound}, so that
      \[ 0 \leq \mathcal L_{\eta,\mu,k}(f_N)(t) \leq |\max v_i|^2\,\left(\sup_t\sup_{i\neq j} \frac{(v_{i}-v_j)_+}{(x_i-x_j)_+}\right)^k\leq C^k, 
\]
independently of $\eta$ and $\mu$.
    \end{example}
Let us start with some basic properties of  $\mathcal L_{\eta,\mu,k}(f_N)(t)$. 
\begin{lemma}
\label{basicproperty}
Assume that $f\in L^\infty([0,\ T],\;M^1(\R^2))$ solves \eqref{basickinetic}, and has bounded moments in $v$ for some $k\geq 0$
\begin{equation}
\sup_{t\in[0,\ T]}\int_{\R^2} |v|^{k+3}\,f(t,x,v)\,dx\,dv<\infty.\label{momentk+3}
\end{equation}
Then for any $\eta,\;\mu>0$, $\mathcal L_{\eta,\mu,k}(f)(t)$ is $BV$ in $t$; in particular $\mathcal L_{\eta,\mu,k}(f)(t)$ is continuous at $a.e.\; t$ and has a left and right trace at every $t$. Furthermore for any $s,\;t$
\begin{equation}
\int_s^t \mathcal L_{\eta,\mu,k}(f)(r)\,dr\longrightarrow \int_s^t \mathcal L_{\eta,0+,k}(f)(r)\,dr, \qquad \mbox{as}\ \mu\rightarrow 0. \label{limitmu0}
\end{equation}
The functional $\int_s^t \mathcal L_{\eta,\mu,k}(f)(r)\,dr$ is also continuous in $f$ and $\int_s^t\mathcal L_{\eta,0+,k}(f)(t)$ is lower semi-continuous in the following sense: If $f_n$ is a sequence of solutions to \eqref{basickinetic} with right-hand sides $m_n\geq 0$ s.t.  	    
\[
\sup_n\,\sup_{t\in[0,\ T]}\int_{\R^2} (|x|^2+|v|^{k+3})\,f_n(t,x,v)\,dx\,dv<\infty,
\]
and $f_n\rightarrow f$ in $w-*\,L^\infty([0,\ T],\;M^1(\R^2))$ then 
\[
\int_s^t \mathcal L_{\eta,\mu,k}(f)(r)\,dr=\lim \int_s^t \mathcal L_{\eta,\mu,k}(f_n)(r)\,dr,\quad \int_s^t \mathcal L_{\eta,0+,k}(f)(r)\,dr=\liminf \int_s^t \mathcal L_{\eta,0+,k}(f_n)(r)\,dr.
\]
\end{lemma}
\begin{proof}
First of all the $\mathcal L_{\eta,\mu,k}(f)$ are bounded by moments of $f$
\[\begin{split}
\mathcal L_{\eta,\mu,k}(f)(t)&\leq \eta^{-k}\,\int_{\R^4} (|v|^{k+2}+|w|^{k+2})\,f(t,x,v)\,f(t,y,w)\,dx\,dy\,dv\,dw\\
&= \eta^{-k}\,\left(\int_{\R^2} |v|^{k+2}\,f(t,x,v)\,dx\,dv\right)^2.
\end{split}
\] 	    
By its definition $\mathcal L_{\eta,\mu,k}(f)(t)$ converges pointwise in $t$ to $\mathcal L_{\eta,0+,k}(f)(t)$. Thus the previous bound implies by dominated convergence that for any $s,\;t$
\[
\int_s^t \mathcal L_{\eta,\mu,k}(f)(r)\,dr\longrightarrow \int_s^t \mathcal L_{\eta,0+,k}(f)(r)\,dr,
\]  	   
as $\mu\rightarrow 0$.	    

Next denoting $f'=f(t,y,w)$, from the equation \eqref{basickinetic} on $f$, since every term in $\mathcal L_{\eta,\mu,k-1}(f)(t)$ is smooth, one has that
	    \begin{equation*}\begin{split}
	      \frac{d}{dt}\mathcal L_{\eta,\mu,k}(f)(t) & = \int \left [ f' \, \partial_t f  + f \, \partial_t  f' \right ]\,\frac{\left (v-w\right )_+^{k+2}}{ \left (x-y + \eta\right )^k}\, \chi_\mu (x-y) \,dv\,dw\,dx\,dy, \\
	       & = \int \left [ f' \Big ( -v \, \partial_x f -\partial_{vv} m\Big ) +f \left ( -w \, \partial_y  f' -\partial_{ww}m')\right ) \right ] \\ 
	       & \quad \quad \ \frac{\left (v-w\right )_+^{k+2}}{ \left (x-y + \eta\right )^k}\, \chi_\mu (x-y) \,dv\,dw\,dx\,dy.
\end{split}
	    \end{equation*}
	    Integrating by part the free transport terms of the last relation, with respect to $x$ and $y$, we obtain that for $k\geq 1$
	    \begin{equation*}\begin{split}
	      \frac{d}{dt}\mathcal L_{\eta,\mu,k}(f)(t) & = \int \bigg \{\left (v-w\right )_+^{k+3} f \, f'  \left [- k\,\frac{\chi_\mu(x-y)}{(x-y+\eta)^{k+1}} + \frac{\chi'_\mu (x-y)}{(x-y+\eta)^{k}} \right ] \\ 
	       &  - (k+1)\,(k+2)\,\left [f' \, m + f \, m'\right ] \frac{\left (v-w\right )_+^{k}}{ \left (x-y + \eta\right )^k}\, \chi_\mu (x-y) \bigg \} \, dv\,dw\,dx\,dy.
	   \end{split}
	    \end{equation*}
Recalling that $m\geq 0$, this leads to 
\begin{equation}
\frac{d}{dt}\mathcal L_{\eta,\mu,k}(f)(t) \leq \int \left (v-w\right )_+^{k+3} f \, f'  \left [ -k\,\frac{\chi_\mu(x-y)}{(x-y+\eta)^{k+1}} + \frac{\chi'_\mu (x-y)}{(x-y+\eta)^{k}} \right ]\,dx\,dy\,dv\,dw,\label{timederivative}
\end{equation}
and hence by \eqref{hypchi}
\[
\frac{d}{dt}\mathcal L_{\eta,\mu,k}(f)(t) \leq \frac{4}{\mu\,\eta^{k}}\,\int |v|^{k+3} f\,dx\,dv,
\]
which is bounded by \eqref{momentk+3}.

On the other hand if $k=0$ by the definition of $\mathcal L_{\eta,\mu,k}(f)$ and with similar calculations
\begin{equation*}\begin{split}
	      \frac{d}{dt}\mathcal L_{\eta,\mu,0}(f)(t) & = -\int_{x\leq 1+y+\eta} \bigg \{\left (v-w\right )_+^{3} f \, f'  \left [ \frac{\chi_\mu(x-y)}{(x-y+\eta)} + \chi'_\mu (x-y)\,\log(x-y+\eta)_- \right ] \\ 
	       &  + 2\,\left [f' \, m + f \, m'\right ] \ind_{v-w\geq 0}\,\log (x-y + \eta)_-\, \chi_\mu (x-y) \bigg \} \, dv\,dw\,dx\,dy.
	   \end{split}
	    \end{equation*}
Note that since $\log x\leq 0$ for $x\leq 1$ and $m\geq 0$, one has similarly in this case
\begin{equation}
\frac{d}{dt}\mathcal L_{\eta,\mu,0}(f)(t) \leq -\int_{x\leq 1+y+\eta}\!\!\!\!\!\! (v-w)_+^{3} f \, f'  \left [ \frac{\chi_\mu(x-y)}{(x-y+\eta)} + \chi'_\mu (x-y)\,\log(x-y+\eta) \right ]\,dx\,dy\,dv\,dw,\label{timederivativek=0}
\end{equation}
leading by \eqref{hypchi} to
\[
\frac{d}{dt}\mathcal L_{\eta,\mu,0}(f)(t) \leq \frac{4\,|\log \eta|}{\mu}\,\int |v|^{3} f\,dx\,dv.
\]
In all cases,  $\mathcal L_{\eta,\mu,k}(f)(t)$ is hence semi-Lipschitz and thus $BV$. 

Consider now any sequence $f_n$ of solutions to Eq. \eqref{basickinetic}. Observe that
\[
\sup_n \int_0^T\int_{\R^2} m_n(dt,dx,dv)=\sup_n \int_{\R^2} |v|^2\,(f_n^0(dx,dv)-f_n(T,dx,dv))<\infty.
\]
Therefore by \eqref{basickinetic}, $\partial_t f_n$ is bounded in $M^1_{loc}([0,\ T]\times\R^2)+L^\infty(W^{-1,1}_x L^1_x)$. That implies that $f_n$ is compact in $L^2([0,\ T])$ with values in some weak space. 

On the other hand the function $(v-w)_+^{k+2}\,(x-y+\eta)^{-k}\,\chi_\mu(x-y)$ is smooth ($C^\infty$) for any $\eta,\;\mu>0$. The uniform control on the moments of $f_n$ then implies that
\[
I_n(t,x,v)=\int_{\R^2} (v-w)_+^{k+2}\,(x-y+\eta)^{-k}\,\chi_\mu(x-y)\,f_n(t,dy,dw)
\] 
is compact in $L^2([0,\ T], C^1_{x,v})$. Therefore we can easily pass to the limit in 
\[
\int_s^t \int_{\R^2} f_n(r,dx,dv)\,I_n(r,x,v)\,dr= \int_s^t\mathcal L_{\eta,\mu,k}(f_n)(r) \,dr.
\]
This obviously cannot work for $L_{\eta,0+,k}(f_n)(t)$. However as $L_{\eta,\mu,k}(f_n)(t)$ is increasing in $\mu$, and by \eqref{limitmu0}
\[
\int_s^t \mathcal L_{\eta,0+,k}(f)(r) \,dr=\sup_{\mu>0}\int_s^t \mathcal L_{\eta,\mu,k}(f)(r) \,dr.
\] 
The supremum of any family of continuous functions is automatically lower semi-continuous thus finishing the proof.
\end{proof}
\subsection{Dissipation properties}
  
Our main goal is to use the dispersive properties of the free transport to bound $\mathcal L_{\eta,0+,k}(f)$ in terms of $\mathcal L_{\eta,0+,k-1}(f)$.
        \begin{theorem}
      \label{thmBoundL}
     Assume that $f\in L^\infty([0,\ T],\;M^1(\R^2))$ solves \eqref{basickinetic}, satisfies \eqref{traceonx} and has bounded moments in $v$ for some $k\geq 0$
\begin{equation}
\sup_{t\in[0,\ T]}\int_{\R^2} |v|^{k+2}\,f(t,x,v)\,dx\,dv<\infty.
\end{equation}
Then for any $\mu,\;\eta>0$, $0\leq s\leq t$ if $k\geq 2$ 
      \begin{equation}\label{generalk}
      k\, \int_s^t \mathcal L_{\eta,0+,k}(f)(r)\,dr+ \mathcal L_{\eta,0+,k-1}(f)(t-) \leq  \mathcal L_{\eta,0+,k-1}(f)(s+)+\frac{2}{\eta^{k-1}}\,\int_s^t\Lambda_{f,k+2}(r)\,dr,
      \end{equation}
and if $k=1$
\begin{equation}\label{k=1}
      \int_s^t \mathcal L_{\eta,0+,1}(f)(r)\,dr+ \mathcal L_{\eta,0+,0}(f)(t-) \leq  \mathcal L_{\eta,0+,0}(f)(s+)+2\,|\log\eta|\,\int_s^t\Lambda_{f,2}(r)\,dr.
      \end{equation}
    \end{theorem}
	  
	  \begin{proof}
The proof is straightforward after Lemma \ref{momentk+3}. We begin by working with $\mathcal L_{\eta,\mu,k}(f)$ for $\mu>0$. Differentiating in time, one again obtains Eq. \eqref{timederivative}, that is
\[
\frac{d}{dt}\mathcal L_{\eta,\mu,k-1}(f)(t) \leq -k\,\mathcal L_{\eta,\mu,k}(f)(t)+\int \left (v-w\right )_+^{k+2} f \, f'  \frac{\chi'_\mu (x-y)}{(x-y+\eta)^{k-1}}\,dx\,dy\,dv\,dw,
\]
for $k\geq 2$ and if $k-1=0$ by \eqref{timederivativek=0},
\[
\frac{d}{dt}\mathcal L_{\eta,\mu,0}(f)(t) \leq -\mathcal L_{\eta,\mu,1}(f)(t)-\int (v-w)_+^{3} f \, f'  \, \chi'_\mu (x-y)\,\log(x-y+\eta)_- \,dx\,dy\,dv\,dw.	   
\]
We now use the property \eqref{hypchi} to bound for $k\geq 2$
\[\begin{split}
&\int \left (v-w\right )_+^{k+2} f \, f'  \frac{\chi'_\mu (x-y)}{(x-y+\eta)^{k-1}}\,dx\,dy\,dv\,dw\leq \frac{2}{\mu\,\eta^{k-1}}\int_{x<y<x+\mu}  (v-w )_+^{k+2} f \, f'\,dx\,dy\,dv\,dw.
\end{split}\]
Therefore integrating in time between $s$ and $t$ the inequality above one has that
\[\begin{split}
\mathcal L_{\eta,\mu,k-1}(f)(t-)-\mathcal L_{\eta,\mu,k-1}(f)(s+)&\leq -k\, \int_s^t \mathcal L_{\eta,\mu,k}(f)(r)\,dr\\
&+\frac{2}{\mu\,\eta^{k-1}}\int_s^t\int_{x<y<x+\mu}  (v-w )_+^{k+2} f \, f'\,dx\,dy\,dv\,dw\,dr.
\end{split}
\]	    
Take the limit $\mu\rightarrow 0$ and observe that by its definition
\[
\limsup_{\mu\rightarrow0}\frac{1}{\mu}\int_s^t\int_{x<y<x+\mu}  (v-w )_+^{k+2} f \, f'\,dx\,dy\,dv\,dw\,dr\leq \int_s^t\Lambda_{f,k+2}(r)\,dr.
\]	    
The passage to the limit in $\mathcal L_{\eta,\mu,k-1}(f)$ and $\int_s^t \mathcal L_{\eta,\mu,k}(f)(r)\,dr$ is provided by Lemma \ref{basicproperty} which concludes the proof in that case. The case $k=1$ is handled similarly.
	  \end{proof}
	  
\subsection{The connection with monokinetic solutions}
It turns out that the functionals $\mathcal L_{\eta,\mu,k}(f)$ can control	the concentration in velocity of a solution to \eqref{basickinetic}. Roughly speaking it is not possible to have a bound on $\mathcal L_{\eta,\mu,k}(f)$ uniform in $\eta$ and $\mu$ if $f$ is not monokinetic. This is due to the fact that $(x-y)^k$ is not integrable if $k\geq 1$ and thus the only way to keep the integral bounded is to have $(v-w)_+$ small if $x$ is close to $y$. 

This is formalized in the following
\begin{proposition}\label{functionalimpliesmonokinetic}
 Assume that $f\in L^\infty([0,\ T],\;M^1(\R^2))$ solves \eqref{basickinetic}, and has bounded $v$-moments for some $k\geq 0$
\[
\sup_{t\in[0,\ T]}\int_{\R^2} |v|^{k+2}\,f(t,x,v)\,dx\,dv<\infty.
\]
Assume moreover that 
\[
\sup_{\mu,\eta} \int_0^T \mathcal L_{\eta,\mu,k}(f)(r)\,dr<\infty.
\]
Then $f$ is monokinetic for $a.e.\;t$: There exist $\rho\in L^\infty([0,\ T],\;M^1(\R))$, $u\in L^\infty([0,\ T],\;L^{k+2}(\rho))$ s.t. for $a.e.\;t$
\[
f(t,x,v)=\rho(t,x)\,\delta(v-u(t,x)).
\] 
\end{proposition}     
\begin{proof}
First of all notice that it is always possible to define $\rho$ and $u$ by
\[
\rho(t,x)=\int_\R f(t,x,dv),\quad \rho\,u(t,x)=\int_\R v\,f(t,x,dv).
\]
One necessarily has that $u\in L^\infty([0,\ T],\;L^{k+2}(\rho))$ because 
\[
u(t,x)=\int_\R v\,\frac{f(t,x,dv)}{\rho(t,x)},
\]
and 
by Jensen inequality
\[
\int_\R |u(t,x)|^{k+2}\,\rho(t,dx)\leq \int_{\R^2} |v|^{k+2}\,f(t,x,dv).
\]
Furthermore by \eqref{basickinetic}, $f$ is $BV$ in time with value in a weak space in $x$ and $v$ (as in the proof of Theorem \ref{thmBoundL}) and using the moments this proves that $\rho$ and $\rho\,u$ are also $BV$ in time.

Radon-Nikodym theorem implies that it is possible to decompose $f$ according to $\rho$
\[
f(t,x,v)=\rho(t,x)\,M(t,x,v),
\]
and the goal is thus to prove that $M$ is concentrated on a Dirac mass. We proceed in two steps by considering the atomic and non-atomic parts of $\rho$. We write accordingly
\[
\rho(t,x)=\sum_{n=1}^\infty \rho_n(t)\,\delta(x-x_n(t))+\tilde\rho(t,x),
\] 
where $\tilde \rho$ does not contain any Dirac mass.

\noindent {\bf Step 1: Control of the non-atomic part.} This part does not require any further use of  Eq. \eqref{basickinetic}. Start by remarking that by Jensen's inequality again
\[
\int_{\R^2} \frac{(u(x)-u(y))_+^{k+2}}{(x-y+\eta)^k}\,\chi_\mu(x-y)\,\rho(t,dx)\,\rho(t,dy) \leq  \mathcal L_{\eta,\mu,k}(f)(t).
\]
Instead of replacing both $v$ and $w$ in $\mathcal L_{\eta,\mu,k}(f)(t)$, it is also possible to use Jensen's inequality to replace only $v$ for instance. Thus
one has as well
\[
\int_{\R^3} \frac{(v-u(y))_+^{k+2}}{(x-y+\eta)^k}\,\chi_\mu(x-y)\,f(t,dx,dv)\,\rho(t,dy) \leq  \mathcal L_{\eta,\mu,k}(f)(t).
\]
Now $(a+b)^k\leq 2^k\,(a^k+b^k)$ and combining the two previous inequalities
\[\begin{split}
&\int_{\R^3} \frac{(v-u(x))_+^{k+2}}{(x-y+\eta)^k}\,\chi_\mu(x-y)\,f(t,dx,dv)\,\rho(t,dy) \\
&\quad=  \int_{\R^3} \frac{(v-u(y)+u(y)-u(x))_+^{k+2}}{(x-y+\eta)^k}\,\chi_\mu(x-y)\,f(t,dx,dv)\,\rho(t,dy)\\
&\quad\leq 2^k\,\int_{\R^3} \frac{(v-u(y))_+^{k+2}}{(x-y+\eta)^k}\,\chi_\mu(x-y)\,f(t,dx,dv)\,\rho(t,dy)\\
&\qquad+2^k\,\int_{\R^2} \frac{(u(x)-u(y))_+^{k+2}}{(x-y+\eta)^k}\,\chi_\mu(x-y)\,\rho(t,dx)\,\rho(t,dy)\\
&\quad\leq2^{k+1}\,\mathcal L_{\eta,\mu,k}(f)(t).
\end{split}
\]
In the left-hand side, only $\rho(t,dy)$ depends on $y$ and this leads to define 
\[
\alpha_{\mu,\eta}=\int_{\R} \frac{\chi_\mu(x)}{(x+\eta)^k}\,dx,\quad K_{\mu,\eta}=\frac{\alpha_{\mu,\eta}^{-1}\,\chi_\mu(x)}{(x+\eta)^k}, \quad\alpha_\eta=\lim_{\mu\rightarrow 0} \alpha_{\mu,\eta},\quad K_\eta=\lim_{\mu\rightarrow 0} K_{\mu,\eta}=\frac{\alpha_{\eta}^{-1}\,\ind_{x>0}}{(x+\eta)^k}.
\]
The previous inequality can be written as
\[
\int_{\R^2} (v-u(x))_+^{k+2}\,K_{\mu,\eta}\star\rho(t,x)\,f(t,dx,dv)\leq \alpha_{\mu,\eta}^{-1}\,2^{k+1}\,\mathcal L_{\eta,\mu,k}(f)(t).
\]
One has that $\lim_{\mu\rightarrow 0} K_{\mu,\eta}\star\rho=\sup_\mu K_{\mu,\eta}\star\rho=K_{\eta}\star\rho$. Note that $K_\eta$ is not continuous and in particular it is defined with $\ind_{x>0}$ and not $\ind_{x\geq 0}$. This makes a difference if $\rho$ contains Dirac masses and as we will see it is the reason why additional calculations are required for the atomic part.

In the meantime integrating in time, taking the supremum in $\mu$ and using the decomposition of $f$, one obtains that
\[
\int_0^T\int_{\R^2} (v-u(t,x))_+^{k+2}\,K_{\eta}\star\rho(t,x)\,\rho(t,dx)\,M(t,x,dv)\,dt\leq \int_0^T\alpha_{\eta}^{-1}\,2^{k+1}\,\mathcal L_{\eta,0+,k}(f)(t)\,dt\longrightarrow 0,
\]
as $\eta\rightarrow 0$, since $x^{-k}$ is not integrable for $k\geq 1$ and thus $\alpha_\eta\rightarrow +\infty$. Therefore for $\rho(t,dx)\,dt$ almost every point $t$ and $x$ s.t.
\begin{equation}
\liminf_{\eta\rightarrow 0}K_\eta\star \rho(t,x)>0,\label{halfdensity}
\end{equation}
then one must have that the support of $M(t,x,.)$ in $v$ is included in $(-\infty,\ u(t,x)]$. However by their definition, one has that for $\rho(t,dx)\,dt$ almost every point $t$ and $x$
\[
\int_\R M(t,x,dv)=u(t,x).
\]
Thus at such points $t$ and $x$ s.t. \eqref{halfdensity} holds, one must have that $M(t,x,v)=\delta(v-u(t,x))$ which is our goal.

In this argument, we treated differently $x$ and $y$ in $\mathcal L_{\eta,0+,k}(f)(t)$. We can make the symmetric argument, deducing that for $\rho(t,dy)\,dt$ almost every point $t$ and $y$ s.t.
\[
\liminf_{\eta\rightarrow 0}\int_\R K_\eta(x-y)\, \rho(t,dx)>0,
\] 
then  the support of $M(t,y,.)$ in $w$ is included in $[u(t,x),\ +\infty)$ and again one must have that $M(t,y,w)=\delta(w-u(t,x))$.

Combining those two arguments, we deduce that  $M(t,x,v)=\delta(v-u(t,x))$ for $\rho(t,dx)\,dt$ almost every point $t$ and $x$ s.t.
\[
\liminf_{\eta\rightarrow 0}\int_\R (K_\eta(x-y)+K_\eta(y-x))\,\rho(t,dy)=\liminf_{\eta\rightarrow 0}\int_{y\neq x}  \frac{\rho(t,dy)}{(|x-y|+\eta)^k}>0.
\] 
We emphasize that $\rho(t,dy)$ is only integrated on $y\neq x$ so that a Dirac mass at $x$ in $\rho$ does not contribute to the previous integral.  Finally
\[
\int_{y\neq x}  \frac{\rho(t,dy)}{(|x-y|+\eta)^k}\geq 2\,\eta^{-1}\int_{B(x,\eta),\;y\neq x} \rho(t,dy),
\] 
yielding
\begin{equation}
M(t,x,v)=\delta(v-u(t,x))\quad\mbox{for}\ \rho\,dt\ a.e.\;t,\,x\ \mbox{s.t.}\ \liminf_{\eta\rightarrow 0} \eta^{-1}\int_{B(x,\eta),\;y\neq x} \rho(t,dy)>0.  \label{conditionmonokinetic}
\end{equation}
To conclude this step, use the classical Besicovitch derivation theorem which implies that for $dt\,\tilde\rho$ $a.e.$ $t,\;x$ then 
\[
\liminf_{\eta\rightarrow 0} \frac{1}{2\,\eta}\,\int_{B(x,\eta),\;y\neq x} \tilde\rho(t,dy)=\liminf_{\eta\rightarrow 0} \frac{1}{2\,\eta}\,\int_{B(x,\eta)} \tilde\rho(t,dy)=\lim_{\eta\rightarrow 0} \frac{1}{2\,\eta}\,\int_{B(x,\eta)} \tilde\rho(t,dy)>0,
\]
as $\tilde\rho$ does not have any Dirac mass. 

This means that for $dt\,\tilde\rho$ $a.e.$ $t,\;x$, $M(t,x,v)=\delta(v-u(t,x))$ and
\[
f(t,x,v)=\tilde\rho(t,x)\,\delta(v-u(t,x))+\sum_{n=1}^\infty \rho_n(t)\,\delta(x-x_n(t))\,M(t,x_n,v).
\]

\medskip

\noindent{\bf Step 2: Control of the atomic part.} As noticed the previous step does not control the atomic part of $f$. Given that $f$ is $BV$ in time, by contradiction if $f$ is not monokinetic at $a.e.\ t$ then there exists $t_0$, $x_0$, $\rho_0>0$ and $M_0(v)\neq \delta(v-u(t_0,x_0))$ s.t.
\[
f(t_0+,x,v)=g+\rho_0\,\delta(x-x_0)\,M_0(v),\quad g\geq0,\quad\int_\R M_0(dv)=1.
\]
The main idea then is to use Eq. \eqref{basickinetic} to show that in that case the atom at $x_0$ has to split at $t>t_+$. The corresponding pieces will now necessarily interact in $\mathcal L_{\eta,0+,k}(f)(t)$, not being at the same point and this will lead to a contradiction. 

Since $M_0$ is not a Dirac mass, it is possible to find two smooth non-negative functions $\varphi_1$ and $\varphi_2$, supported on distinct intervals $I_1$ and $I_2$ s.t.
		  \begin{equation}
		    \label{defI1I2}
		    \inf \left \{ (v - w)_+ \,: \, v \in I_1, w \in I_2 \right \} \geq C_* >0,
		  \end{equation}
and
\[
\inf_{i=1,\;2}\int_\R M_0(dv)\,\varphi_i(v)\geq \frac{1}{3}.
\]
Denote these intervals as 
		    $I_i := \left [\underline{v_i}, \, \overline{v_i} \right ]$, and calculate using Eq. \eqref{basickinetic} for $t>t_0$		  
		  \begin{align*}
		    &\frac{d}{dt} \int_{x_0 + \underline{v_i} \, (t-t_0)}^{x_0 + \overline{v_i} \, (t-t_0)} \int_\R  \varphi_i(v) \, f(t,dx,dv)  = \int_{\R} \overline{v_i}\,\varphi_i \,f\left (t,x_0\!+\! \overline{v_i} \, (t\!-\!t_0), dv \right )\\
&\qquad - \int_{\R} \underline{v_i}\,\varphi_i \,f\left (t,x_0\!+\! \underline{v_i} \, (t\!-\!t_0), dv \right)- \int_{x_0 + \underline{v_i} \, (t-t_0)}^{x_0 + \overline{v_i} \, (t-t_0)} \int_\R  \varphi_i(v) \, v\cdot\partial_x f(t,dx,dv)\\
&\qquad -  \int_{x_0 + \underline{v_j} \, t}^{x_0 + \overline{v_j} \, t} \int_v  \varphi_i(v) \, \partial_{vv} \mu(t,dx,dv). 
		  \end{align*}  
Integrating by part the term in $v\partial_xf$, we find
\begin{align*}
		    &\frac{d}{dt} \int_{x_0 + \underline{v_i} \, (t-t_0)}^{x_0 + \overline{v_i} \, (t-t_0)} \int_\R  \varphi_i(v) \, f(t,dx,dv)  = \int_{\R} (\overline{v_i}-v)\,\varphi_i \,f\left (t,x_0\!+\! \overline{v_i} \, (t\!-\!t_0), dv \right )\\
&\qquad \int_{\R} (v-\underline{v_i})\,\varphi_i \,f\left (t,x_0\!+\! \underline{v_i} \, (t\!-\!t_0), dv \right) -  \int_{x_0 + \underline{v_j} \, t}^{x_0 + \overline{v_j} \, t} \int_v  \varphi_i(v) \, \partial_{vv} m(t,dx,dv). 
		  \end{align*}  
Since $\varphi_i$ is supported on the interval $I_i$, we have there that $\overline{v_i}-v\geq 0$ and $v-\underline{v_i}\geq 0$ so integrating between $t_0$ and $t$
\[
\int_{x_0 + \underline{v_i} \, (t-t_0)}^{x_0 + \overline{v_i} \, (t-t_0)} \int_\R  \varphi_i(v) \, f(t,dx,dv) \geq \rho_0\,\int_\R \varphi_i(v)\,M_0(dv)-\int_{t_0+}^t\int_{\R^2} |\partial_{vv}\varphi_i|\,m(dt,dx,dv),
\]
and hence for some constant $C>0$
\[
\int_{x_0 + \underline{v_i} \, (t-t_0)}^{x_0 + \overline{v_i} \, (t-t_0)} \int_\R  \varphi_i(v) \, f(t,dx,dv) \geq \frac{\rho_0}{3}-C\,\int_{t_0+}^t\int_{\R^2} m(dt,dx,dv).
\]
The measure $m$ has finite total mass as it can be checked by integrating Eq. \eqref{basickinetic} against $|v|^2$
\[
\int_0^T\int_{\R^2} m(dt,dx,dv)\leq \int_{\R^2} |v|^2\,f^0(dx,dv)<\infty. 
\]		 
In particular this implies that
\[
\int_{t_0+}^t\int_{\R^2} m(dt,dx,dv)\longrightarrow 0,\quad \mbox{as}\ t\rightarrow t_0,
\]
and that there exists a critical time $t_c>t_0$ s.t.
 \[
\int_{t_0+}^{t_c}\int_{\R^2} m(dt,dx,dv)\leq \frac{\rho_0}{6\,C}.
\]
Consequently for any $t_0<t<t_c$
\begin{equation}
\inf_{i=1,\;2}\int_{x_0 + \underline{v_i} \, (t-t_0)}^{x_0 + \overline{v_i} \, (t-t_0)} \int_\R  \varphi_i(v) \, f(t,dx,dv) \geq \frac{\rho_0}{6}.\label{massonintervals}
\end{equation}
Inserting this decomposition in $\mathcal L_{\eta,\mu,k}(f)(t)$
\[\begin{split}
&\int_{t_0}^{t_c}\mathcal L_{\eta,0+,k}(f)(t)\,dt\\
&\ \geq \int_{t_0}^{t_c}\int_{x_0 + \underline{v_1} \, (t-t_0)}^{x_0 + \overline{v_1}}\int_{x_0 + \underline{v_2} \, (t-t_0)}^{x_0 + \overline{v_2}}\int_{\R^2}  \frac{(v-w)_+^{k+2}}{( x-y + \eta)^k}\, \chi_\mu (x-y) \,\varphi_1(v)\,\varphi_2(w) f(t,dx,dv) f(t,dy,dw)\,dt\\
&\ \geq \int_{t_0}^{t_c}\int_{x_0 + \underline{v_1} \, (t-t_0)}^{x_0 + \overline{v_1}}\int_{x_0 + \underline{v_2} \, (t-t_0)}^{x_0 + \overline{v_2}}\int_{\R^2}  \frac{C_*^{k+2}}{( x-y + \eta)^k}\, \chi_\mu (x-y) \,\varphi_1(v)\,\varphi_2(w) f(t,dx,dv) f(t,dy,dw)\,dt,
\end{split}
\]
by \eqref{defI1I2} since $\varphi_i$ is supported in $I_i$.

If $x\in [x_0+\underline{v_1} \, (t-t_0),\ x_0+\overline{v_1} \, (t-t_0)]$ and $y\in [x_0+\underline{v_2} \, (t-t_0),\ x_0+\overline{v_2} \, (t-t_0)]$, then 
by \eqref{defI1I2}
\[
x-y\geq (\underline{v_1}-\overline{v_2})\,(t-t_0)\geq C_*\,(t-t_0).
\]
 Therefore
\[\begin{split}
&\int_{t_0}^{t_c}\mathcal L_{\eta,0+,k}(f)(t)\,dt\\
&\ \geq \int_{t_0+\mu/C_*}^{t_c}\int_{x_0 + \underline{v_1} \, (t-t_0)}^{x_0 + \overline{v_1}}\int_{x_0 + \underline{v_2} \, (t-t_0)}^{x_0 + \overline{v_2}}\int_{\R^2}  \frac{C_*^{k+2}}{(C_*\,(t-t_0) + \eta)^k}\, \,\varphi_1(v)\,\varphi_2(w) f(t,dx,dv)\, f(t,dy,dw)\,dt\\
&\ \geq \int_{t_0+\mu/C_*}^{t_c} \frac{C_*^{k+2}}{(C_*\,(t-t_0) + \eta)^k}\,\frac{\rho_0^2}{36},
\end{split}
\] 
by \eqref{massonintervals}.
Finally if $k>1$ this implies that
\[
\int_{t_0}^{t_c}\mathcal L_{\eta,0+,k}(f)(t)\,dt\geq \frac{\rho_0^2}{36}\,\frac{C_*^{k+2}}{k\,(\mu + \eta)^{k-1}},
\]		  
and if $k=1$
\[
\int_{t_0}^{t_c}\mathcal L_{\eta,0+,k}(f)(t)\,dt\geq - \frac{\rho_0^2}{36}\,C_*^{3}\,\log(\mu + \eta).
\]
In both cases, one obtains that
\[
\sup_{\eta,\mu} \int_{0}^{T}\mathcal L_{\eta,0+,k}(f)(t)\,dt=\infty,
\]
which is a contradiction.
\end{proof}	    
	    
  \section{Hydrodynamic Limit: Proof of Theorem \ref{mainresult}} 
    \label{secHydro}
  	  
		\subsection{A general Hydrodynamic Limit }
We prove here a more general version of Theorem \ref{mainresult} which can apply to many different systems. 
\begin{theorem}
Assume that one has a sequence $f_\eps\in L^\infty([0,\ T],\ M^1(\R^2))$ of solutions to \eqref{basickinetic} with mass $1$ for a corresponding sequence of non negative measures $m_\eps$. Assume that all $v$-moments of $f_\eps$ are bounded uniformly in $\eps$: For any $k$
\[
\sup_\eps \sup_{t\in[0,\ T]} \int_{\R^2} |v|^{k}\,f_\eps (t,dx,dv)<\infty,
\]
together with one moment in $x$, for instance
\[
\sup_\eps \sup_{t\in[0,\ T]} \int_{\R^2} |x|^2\,f_\eps(t,dx,dv)<\infty.
\]
Assume moreover that $f_\eps$ satisfies the condition \eqref{traceonx}
with
\begin{equation}
\int_0^T \Lambda_{f_\eps,k}(t)\,dt\longrightarrow 0,\quad\mbox{as}\ \eps\rightarrow 0\ \mbox{for any fixed}\ k,\label{lambdavanish}
\end{equation}
with finally that 
\begin{equation}
\sup_\eps \sup_{\eta,\mu}\,\mathcal L_{\eta,\mu,0}(f_\eps)(t=0)<\infty.\label{L0bounded}
\end{equation}
Then any weak-* limit $f$ of $f_\eps$ solves the sticky particles dynamics in the sense that $\rho=\int_\R f(t,x,dv)$ and $j=\int_\R v\,f(t,x,dv)=\rho\,u$ are a distributional solution to the pressureless system  \eqref{sysSticky} while $u$ has the Oleinik property for any $t>0$
\begin{equation}
u(t,x)-u(t,y)\leq \frac{x-y}{t},\quad\mbox{for } \rho\;a.e.\ x\geq y. \label{ineqOleinik}
\end{equation}
\label{generalhydro}
\end{theorem}
\begin{remark} \label{rkuniqueness}
	As already mentioned in the introduction,	it is known from \cite{BouchutJames:1999, HuangWang:2001} that there exists a unique solution $(\rho, u)$ to the pressureless Euler equations \eqref{sysSticky} (called the \emph{entropy} solution) under the so-called \emph{Oleinik condition} \eqref{ineqOleinik} for any $t>0$
and if the measure $\rho u^2$ weakly converges to $\rho_{in} u_{in}^2$ as $t$ goes to $0$. Therefore once $f$ is known in Theorem \ref{generalhydro} at some time $t_0$, it is necessarily unique after that time $t_0$. The only problem for uniqueness can occur at $t=0$. This can be remedied if the initial data is well prepared for example
\begin{equation}
f^0(x,v)=\rho^0(x)\,\delta(v-u^0(x)),\quad u^0\,\mbox{Lipschitz}.\label{wellprepared}
\end{equation}
\end{remark}

\begin{proof}
We divide it in distinct steps: First passing to the limit in $f_\eps$ and its moments. Then proving that $f$ is monokinetic which implies that $\rho,\;j$ solve the pressureless system \eqref{sysSticky} and finally obtain the \emph{Oleinik condition} \eqref{ineqOleinik}.

\medskip

\noindent{\bf Step 1: Extracting limits.} First of all, since the total mass is $1$ at any $t$, then the sequence $f_\eps$ is uniformly bounded in $L^\infty([0,\ T],\ M^1(\R^2))$. It is possible to extract a subsequence, still denoted $f_\eps$ for simplicity, that converges to some $f$ in the appropriate weak-* topology: For any $\phi\in L^1([0,\ T], C_c(\R^2))$,
\[
\int_0^T\int_{\R^2} \Phi(t,x,v)\,f_\eps(t,dx,dv)\,dt\longrightarrow \int_0^T\int_{\R^2} \Phi(t,x,v)\,f(t,dx,dv)\,dt.
\]
Since moments up to order at least $3$ of $f_\eps$ are uniformly bounded in $\eps$, then it is also possible to pass to the limit in moments of $f_\eps$ and
\[\begin{split}
&\rho_\eps=\int_\R f_\eps(t,x,dv)\rightarrow \rho=\int_\R f(t,x,dv), \quad j_\eps=\int_\R v\,f_\eps(t,x,dv)\rightarrow j=\int_\R v\,f(t,x,dv),\\
&E_\eps=\int_\R v^2\,f_\eps(t,x,dv)\rightarrow E=\int_\R v^2\,f(t,x,dv),
\end{split}\] 
in the weak-* topology of $L^\infty([0,\ T],\ M^1(\R))$.

Multiplying Eq. \eqref{basickinetic} by $|v|^2$ one finds that
\[
\sup_\eps \int_0^T\int_{\R^2} m_\eps(dt,dx,dv)\leq \sup_\eps\int_\R v^2\,f_\eps(t=0,x,dv)<\infty.
\]
Therefore one may further extract a converging subsequence $m_\eps\rightarrow m\geq 0$ in the weak-* topology of $M^1([0,\ T]\times(\R))$. 

This proves that $f$ and $m$ still solve \eqref{basickinetic}. From the bounded moments of $f$, one may integrate this system against $1$ first and $v^2$ second to find the system
\begin{equation}
\begin{split}
&\partial_t \rho+\partial_x j=0,\\
&\partial_t j+\partial_x E=0.
\end{split}\label{pressureless}
\end{equation}

\medskip

\noindent{\bf Step 2: $f$ is monokinetic.} We now apply Theorem \ref{thmBoundL} to $f_\eps$ for $k=1$ and find from \eqref{generalk} that
\[
\int_0^T \mathcal L_{\eta,0+,1}(f_\eps)(t)\,dt\leq \mathcal L_{\eta,0+,0}(f_\eps)(t=0)+2\,|\log \eta|\,\int_0^T\Lambda_{f_\eps,2}(t)\,dt. 
\]
This means in particular that for any $\mu>0$
\[
\int_0^T \mathcal L_{\eta,\mu,1}(f_\eps)(t)\,dt\leq \mathcal L_{\eta,0+,0}(f_\eps)(t=0)+2\,|\log \eta|\,\int_0^T\Lambda_{f_\eps,2}(t)\,dt. 
\]
We use Lemma \ref{basicproperty} on the sequence $f_\eps$ to obtain that for any $\mu>0$ and $\eta>0$
\[
\int_0^T \mathcal L_{\eta,\mu,1}(f_\eps)(t)\,dt\longrightarrow \int_0^T \mathcal L_{\eta,\mu,1}(f)(t)\,dt.
\]
By the assumptions of Theorem \ref{generalhydro} we also have that $\int_0^T\Lambda_{f_\eps,2}(t)\,dt\rightarrow 0$ and that $C:=\sup_{\eps,\eta}\mathcal L_{\eta,0+,0}(f_\eps)(t=0)<\infty$. Thus
\[
\int_0^T \mathcal L_{\eta,\mu,1}(f)(t)\,dt\leq C,
\]  
and in particular
\[
\sup_{\mu,\eta} \int_0^T \mathcal L_{\eta,\mu,1}(f)(t)\,dt<\infty.
\]
We may now apply Prop. \ref{functionalimpliesmonokinetic} which implies that $f$ is monokinetic, that is $f=\rho(t,x)\,\delta(v-u(t,x))$ while $u$ satisfies that for any $k$
\[
\sup_{[0,\ T]}\int_\R |u(t,x)|^{k}\,\rho(t,dx)<\infty.
\]
Therefore one automatically has that $j=\rho\,u$ and $E=\rho\,|u|^2$. From system \eqref{pressureless}, $\rho$ and $\rho\,u$ solve the pressureless gas dynamics \eqref{sysSticky}.

\medskip

\noindent{\bf Step 3: The Oleinik condition.} We only have to show that $u$ is semi-Lipschitz in the sense of \eqref{ineqOleinik}. Since all moments of $f_\eps$ are bounded, we may apply Theorem \ref{thmBoundL} to $f_\eps$ for any $k$ of which we repeat the conclusion
\begin{equation}
\mathcal L_{\eta,0+,k-1}(f_\eps)(t)+k\,\int_s^t \mathcal L_{\eta,0+,k}(f_\eps)(r)\,dr\leq \mathcal L_{\eta,0+,k-1}(f_\eps)(s)+\frac{2}{\eta^{k-1}}\,\int_s^t\Lambda_{f_\eps,k+2}(r)\,dr. \label{debutgronwall}
\end{equation}

Observe that by a simple H\"older inequality
\[\begin{split}
&\mathcal L_{\eta,0+,k-1}(f_\eps)(t)= \int_{\R^4} \ind_{x>y}\, \frac{(v-w)_+^{k-1}}{(x-y+\eta)^{k-1}}\,(v-w)_+^2\, f_\eps\,f_\eps'\\
&\leq \left(\int_{\R^4} \ind_{x>y}\, \frac{(v-w)_+^{k}}{(x-y+\eta)^{k}} \,(v-w)_+^2\, f_\eps\,f_\eps'\right)^{(k-1)/k}\,\left(\int_{\R^4} (v-w)_+^2 f_\eps\,f_\eps'\right)^{1/k}.\\
\end{split}\] 
Therefore since $\int |v|^2\,f_\eps(dx,dv)$ is uniformly bounded in $\eps$, for some uniform constant $C$ one has that
\[
\mathcal L_{\eta,0+,k-1}(f_\eps)(t)\leq C^{1/k}\,(\mathcal L_{\eta,0+,k}(f_\eps)(t))^{(k-1)/k},
\]
which from the inequality \eqref{debutgronwall} leads to for $a.e.\ s<t$
\[
\mathcal L_{\eta,0+,k-1}(f_\eps)(t)+k\,C^{-\frac1{k-1}}\,\int_s^t (\mathcal L_{\eta,0+,k-1}(f_\eps)(r))^{\frac{k}{k-1}}\,dr\leq  \mathcal L_{\eta,0+,k-1}(f_\eps)(s)+\frac{2}{\eta^{k-1}}\,\int_s^t\Lambda_{f_\eps,k+2}(r)\,dr.
\]
This is now a closed inequality on $\mathcal L_{\eta,0+,k-1}(f_\eps)(t)$.  In order to derive a bound in a simple manner, assume momentarily that $\Lambda_{f_\eps,k+2}$ is $L^\infty$ in time, or more precisely approximate it by such a bounded function. Then the inequality would imply that $\mathcal L_{\eta,0+,k-1}(f_\eps)$ is Lipschitz and could be rewritten in the more direct form
\[
\frac{d}{dt} \mathcal L_{\eta,0+,k-1}(f_\eps)(t)\leq -k\,C^{-\frac1{k-1}} (\mathcal L_{\eta,0+,k-1}(f_\eps)(r))^{\frac{k}{k-1}}+\frac{2}{\eta^{k-1}}\,\Lambda_{f_\eps,k+2}(t).
\]
Introduce the intermediary quantity $M(t)=t^{k-1}\, \mathcal L_{\eta,0+,k-1}(f_\eps)(t)$ which satisfies now
\[
\frac{dM}{dt}\leq (k-1)\,\frac{M-C^{-\frac1{k-1}}\,M^{1+\frac{1}{k-1}}}{t} +\frac{2\,t^{k-1}}{\eta^{k-1}}\,\Lambda_{f_\eps,k+2}(t).
\]
At a given point $t$, either $M\leq C$ or 
\[
\frac{dM}{dt}\leq \frac{2\,t^{k-1}}{\eta^{k-1}}\,\Lambda_{f_\eps,k+2}(t).
\]
Therefore obviously
\[
M(t)\leq C+\frac{2\,T^{k-1}}{\eta^{k-1}}\,\int_0^T \Lambda_{f_\eps,k+2}(r)\,dr.
\]
This final bound now only depends on the $L^1$ norm of $\Lambda_{f_\eps,2}$ (and thus is independent of the chosen approximation of $\Lambda_{f_\eps,k+2}$) leading to the inequality
\[
t^{k-1}\,\mathcal L_{\eta,0+,k-1}(f_\eps)(t)\leq C+\frac{2\,T^{k-1}}{\eta^{k-1}}\,\int_0^T \Lambda_{f_\eps,k+2}(r)\,dr.
\]
Integrating this inequality between  $0$ and $T$ and recalling that
$\mathcal L_{\eta,\mu,k-1}(f_\eps)(t)\leq \mathcal L_{\eta,0+,k-1}(f_\eps)(t)$, one obtains that for any $\mu>0$ and $\eta>0$, one has
\[
\int_0^T r^{k-1}\,\mathcal L_{\eta,\mu,k-1}(f_\eps)(r)\,dr\leq C\,T+\frac{2\,T^{k}}{\eta^{k-1}}\,\int_0^T \Lambda_{f_\eps,k+2}(r)\,dr. 
\]
Because of $r^{k-1}$ it is now possible to pass to the limit as $\eps \rightarrow 0$ by Lemma \ref{basicproperty}. Recall that from the assumption of Theorem \ref{generalhydro}, $\int_0^T \Lambda_{f_\eps,k+2}(r)\,dr\rightarrow 0$ to obtain
\[
\int_0^T r^{k-1}\,\mathcal L_{\eta,\mu,k-1}(f)(r)\,dr\leq C\,T.
\]
Take the supremum in $\mu$ to find from Lemma \ref{basicproperty} that
\[
\int_0^T r^{k-1}\,\mathcal L_{\eta,0+,k-1}(f)(r)\,dr\leq C\,T,
\]
or recalling the definition of $\mathcal L_{\eta,0+,k-1}(f)$ and the fact that $f$ is monokinetic
\[
\int_0^T\int_{\R^2} \ind_{x>y}\,(u(t,x)-u(t,y))_+^2\,\left(t\,\frac{(u(t,x)-u(t,y))_+}{(x-y+\eta)}\right)^{k-1}\,\rho(t,dx)\,\rho(t,dy)\,dt\leq C\,T. 
\]
For a fixed $\eta$, take the limit $k\rightarrow\infty$ in this inequality.
The only possibility for the left-hand side to remain bounded is that on the support of $\ind_{x>y}\,\rho(t,x)\,\rho(t,y)$, one has that
\[
t\,\frac{(u(t,x)-u(t,y))_+}{(x-y+\eta)}\leq 1.
\]
This is uniform in $\eta$ and thus passing finally to the limit $\eta\rightarrow 0$, one recovers the Oleinik bound \eqref{ineqOleinik}.
\end{proof}	  
		  
\subsection{Proof of Theorem \ref{mainresult}}
Let us start by checking that $f_\eps$ is a solution to Eq. \eqref{basickinetic}. Given that $f_\eps$ solves Eq. \eqref{eqBoltzEPS}, this is equivalent to showing that for any $\alpha$ and any $f$ the collision kernel $\mathcal Q_\alpha(f,f)$ can be represented as $-\partial_{vv} m$ for some non-negative measure $m$. 

Thus we have to show that
\[
\int_\R \mathcal Q_\alpha(f,f)\,dv=0,\quad \int_\R v\,\mathcal Q_\alpha(f,f)\,dv=0,
\]
which is just the conservation of mass and momentum, 
and that for any $\psi(v)$ with $\partial_{vv}\psi\geq 0$, that is $\psi$ convex,
\[
\int_\R\psi(v)\,\mathcal Q_\alpha(f,f)\,dv\leq 0.
\]
This is a consequence of the weak formulation of the operator \eqref{defBoltzweak}, which reads as we recall for any smooth test function $\psi$
\begin{equation}
\int_\R\psi(v)\,\mathcal Q_\alpha(f,f)\,dv=\frac{1}{2}\int_{\R^2}|v-v_*|\,f(v_*)\,f(v)\,(\psi(v')+\psi(v'_*)-\psi(v_*)-\psi(v))\,dv\,dv_*.
\end{equation}
Now rewriting  $v'$ and $v'_*$ 
\[\begin{split}
\psi(v')+\psi(v'_*)-\psi(v_*)-\psi(v)&=\psi\left(\frac{1\!+\!\alpha}2\,v +\frac{1\!-\!\alpha}{2}\,v_*\right)+\psi\left(\frac{1\!-\!\alpha}2\,v +\frac{1\!+\!\alpha}{2}\,v_*\right)-\psi(v_*)-\psi(v)\\
&\leq 0,
\end{split}
\]
if $\psi$ convex for $\alpha<1$. 

This implies that propagating moments is easy
\[
\frac{d}{dt}\int_{\R^2} |v|^k\,f_\eps\,dx\,dv=\frac1\eps\int_{\R^2} |v|^k\,Q_\alpha(f_\eps,f_\eps)\,dx\,dv\leq 0.
\]
This immediately proves that
\[
\sup_\eps \sup_{t\in[0,\ T]} \int_{\R^2} |v|^k f_\eps(t,x,v)\,dx\,dv\leq
 \sup_\eps \int_{\R^2} |v|^k f_\eps^0(x,v)\,dx\,dv<\infty.
\]
Next note that
\[
\frac{d}{dt}\int_{\R^2} |x|^2\,f_\eps(t,x,v)\,dx\,dv=2\int_{\R^2} x\cdot v\,f_\eps(t,x,v)\,dx\,dv\leq \int_{\R^2} (|x|^2+|v|^2)\,f_\eps(t,x,v)\,dx\,dv,
\]
so that
\[
\sup_\eps \sup_{t\in[0,\ T]} \int_{\R^2} |x|^2 f_\eps(t,x,v)\,dx\,dv\leq
 e^T\,\sup_\eps \int_{\R^2} (|x|^2+|v|^2)\, f_\eps^0(x,v)\,dx\,dv<\infty.
\]
In addition the dissipation term from the $v$-moments actually leads to a control on  $\Lambda_{f_\eps,k}(t)$. Since we assumed that $f_\eps\in L^\infty([0,\ T],\ L^p(\R^2))$ for $p>2$ and every moment of $f_\eps$ is bounded then for any fixed $v$, then
\[
\int_{\R} (v-w)_+^k\,f_\eps(t,x,w)\,dw
\]
is bounded in $L^2([0,\ T]\times\R)$ and by standard approximation by convolution
\[
\int_{\R^2} \frac{\ind_{x<y<\delta}}{\delta} (v-w)_+^k\,f_\eps(t,y,w)\,dy\,dw-\int_{\R} (v-w)_+^k\,f_\eps(t,x,w)\,dw\longrightarrow 0,
\]
in $L^2([0,\ T]\times\R)$ as $\delta\rightarrow 0$. Of course this convergence only holds for a fixed $\eps$ (and is not in principle uniform in $\eps$). But for a fixed $\eps$, it now directly implies that for $a.e.\;t$
\[
\Lambda_{f_\eps,k}=\int_{\R^3} (v-w)_+^k\,f_\eps(t,x,v)\,f_\eps(t,x,w)\,dx\,dv\,dw.
\]
As suggested in the introduction for the energy, $k=2$, this term is then controlled by the dissipation of the moment of order $k$. More precisely if $\psi(v)=|v|^k$ then for some $C_k>0$ 
\[\begin{split}
&\psi(v')+\psi(v'_*)-\psi(v_*)-\psi(v)\geq \frac{|v-v_*|^k}{C_k}.
\end{split}
\]
Therefore 
\[
\int_0^T\Lambda_{f_\eps,k}\leq C_k\,\eps\int_0^T\int_{\R^2} \mathcal Q_\alpha(f_\eps,f_\eps)\,dx\,dv\,dt\leq C_k\,\eps\,\int_{\R^2} |v|^k\,f_\eps^0(x,v)\,dx\,dv\longrightarrow 0,
\]
as $\eps\rightarrow 0$.

The last assumptions of Theorem \ref{generalhydro} to check is a bound $\mathcal L_{\eta,\mu,0}(f_\eps)(t=0)$ uniformly in $\eps,\;\eta,\;\mu$.
This follows from the uniform $L^p$ bound on $f_\eps^0$ through a straightforward H\"older estimate to compensate for the $\log$ singularity. Denote $q=\frac{1+p}{2}$ and $q^*$ s.t. $1/q^*=1-1/q$
\[\begin{split}
\mathcal L_{\eta,\mu,0}(f_\eps)(t=0)&=-\int_{\R^4} \chi_\mu(x-y)\,|v-w|_+^2\,\log (x-y+\eta)_-\, f_\eps^0(x,v)\,f_\eps^0(y,w)\,dx\,dy\,dv\,dw\\
&\leq 2\,\int_{\R^4} \ind_{|x-y|\leq 2}\,|\log |x-y||\,|v|^2\,f_\eps^0(x,v)\,f_\eps^0(y,w))\,dx\,dy\,dv\,dw\\
&\leq 2\,\int_{\R^2} |v|^2 f_\eps^0(x,v)\,\left(\int_{x-2}^{x+2}\int_\R \frac{|\log |x-y||^{q^*}}{1+|w|^2}\,dy\,dw\right)^{1/q^*}\\
&\qquad\qquad\qquad\qquad\qquad\left(\int_{\R^2} (1+|w|^2)\, |f_\eps^0(y,w)|^q\,dy\,dw \right)^{1/q}\,dx\,dv\\
&\leq C\,\int_{\R^2} |v|^2 f_\eps^0(x,v)\,dx\,dv\,\left(\int_{\R^2} (1+|w|^2)\, |f_\eps^0(y,w)|^q\,dy\,dw \right)^{1/q},
\end{split}
\]
since $|\log x|^l$ is integrable at $0$ for any $l>0$. Finally by Cauchy-Schwartz
\[
\int_{\R^2} (1+|w|^2)\, |f_\eps^0(y,w)|^q\,dy\,dw\leq \left(\int_{\R^2} (1+|w|^2)^2\, f_\eps^0(y,w)\,dy\,dw\right)^{1/2}\,\left(\int_{\R^2} |f_\eps^0(y,w)|^p\,dy\,dw\right)^{1/2},
\]
which gives
\[
\mathcal L_{\eta,\mu,0}(f_\eps)(t=0)\leq C\,\|f_\eps^0\|_{L^p}^{p/2}\,\left(1+\int_{\R^2} |v|^4\,f_\eps^0(x,v)\,dx\,dv\right)^{3/2},
\]
and the uniform bound.

Since the sequence $f_\eps$ satisfies all the assumptions of Theorem \ref{generalhydro}, its conclusions apply thus proving Theorem \ref{mainresult}.
  \bibliographystyle{acm}
  \bibliography{biblioJR}

\begin{thebibliography}{10}

\bibitem{Alonso:2009}
{\sc Alonso, R.~J.}
\newblock {Existence of Global Solutions to the Cauchy Problem for the
  Inelastic Boltzmann Equation with Near-vacuum Data}.
\newblock {\em Indiana Univ. Math. J. 58}, 3 (2009), 999--1022.

\bibitem{AlonsoLods:2010}
{\sc Alonso, R.~J., and Lods, B.}
\newblock {Free Cooling and High-Energy Tails of Granular Gases with Variable
  Restitution Coefficient}.
\newblock {\em SIAM J. Math. Anal. 42}, 6 (2010), 2499--2538.

\bibitem{AlonsoLods:2013}
{\sc Alonso, R.~J., and Lods, B.}
\newblock {Two proofs of Haff’s law for dissipative gases: the use of entropy
  and the weakly inelastic regime}.
\newblock {\em Journal of Mathematical Analysis and Applications 397}, 1
  (2013), 260--275.

\bibitem{BenedettoCagliotiGolsePulvirenti:1999}
{\sc Benedetto, D., Caglioti, E., Golse, F., and Pulvirenti, M.}
\newblock {A hydrodynamic model arising in the context of granular media}.
\newblock {\em Computers {\&} Mathematics with Applications 38}, 7-8 (oct
  1999), 121--131.

\bibitem{BenedettoCagliotiPulvirenti:97}
{\sc Benedetto, D., Caglioti, E., and Pulvirenti, M.}
\newblock {A One-dimensional Boltzmann Equation with Inelastic Collisions}.
\newblock {\em Rend. Sem. Mat. Fis. Milano LXVII\/} (1997), 169--179.

\bibitem{BenedettoPulvirenti:2002}
{\sc Benedetto, D., and Pulvirenti, M.}
\newblock {On the one-dimensional Boltzmann equation for granular flows}.
\newblock {\em M2AN Math. Model. Numer. Anal. 35}, 5 (Apr. 2002), 899--905.

\bibitem{BiCrPa:2006}
{\sc Biryuk, A., Craig, W., and Panferov, V.}
\newblock Strong solutions of the {B}oltzmann equation in one spatial
  dimension.
\newblock {\em C. R. Math. Acad. Sci. Paris 342}, 11 (2006), 843--848.

\bibitem{bony:1987}
{\sc Bony, J.-M.}
\newblock Solutions globales born\'ees pour les mod\`eles discrets de
  l'\'equation de {B}oltzmann, en dimension {$1$} d'espace.
\newblock In {\em Journ\'ees ``\'{E}quations aux deriv\'ees partielles''
  ({S}aint {J}ean de {M}onts, 1987)}. \'Ecole Polytechnique, Palaiseau, 1987.
\newblock Exp.\ No.\ XVI, 10 pp.

\bibitem{BouchutJames:1999}
{\sc Bouchut, F., and James, F.}
\newblock Duality solutions for pressureless gases, monotone scalar
  conservation laws, and uniqueness.
\newblock {\em Comm. Partial Diff. Eq. 24}, 11-12 (1999), 2173--2189.

\bibitem{boudin:2000}
{\sc Boudin, L.}
\newblock {A Solution with Bounded Expansion Rate to the Model of Viscous
  Pressureless Gases}.
\newblock {\em SIAM Journal on Mathematical Analysis 32}, 1 (2000), 172--193.

\bibitem{BrenierGrenier:1998}
{\sc Brenier, Y., and Grenier, E.}
\newblock Sticky particles and scalar conservation laws.
\newblock {\em SIAM J. Numer. Anal. 35}, 6 (1998), 2317--2328 (electronic).

\bibitem{brilliantov:2004}
{\sc Brilliantov, N., and P\"{o}schel, T.}
\newblock {\em {Kinetic {T}heory of {G}ranular {G}ases}}.
\newblock Oxford University Press, USA, 2004.

\bibitem{Cercignani:1992}
{\sc Cercignani, C.}
\newblock A remarkable estimate for the solutions of the {B}oltzmann equation.
\newblock {\em Appl. Math. Lett. 5}, 5 (1992), 59--62.

\bibitem{CIP:94}
{\sc Cercignani, C., Illner, R., and Pulvirenti, M.}
\newblock {\em The Mathematical Theory of Dilute Gases}, vol.~106 of {\em
  Applied Mathematical Sciences}.
\newblock Springer-Verlag, New York, 1994.

\bibitem{ChertockKurganovRykov:2007}
{\sc Chertock, A., Kurganov, A., and Rykov, Y.}
\newblock A new sticky particle method for pressureless gas dynamics.
\newblock {\em SIAM J. Numer. Anal. 45}, 6 (2007), 2408---2441 (electronic).

\bibitem{ERykovSinai:1996}
{\sc E, W., Rykov, Y.~G., and Sinai, Y.~G.}
\newblock Generalized variational principles, global weak solutions and
  behavior with random initial data for systems of conservation laws arising in
  adhesion particle dynamics.
\newblock {\em Commun. Math. Phys. 177}, 2 (1996), 349--380.

\bibitem{GolseStRaymond2}
{\sc Golse, F., and Saint-Raymond, L.}
\newblock {The Navier-Stokes limit of the Boltzmann equation for bounded
  collision kernels}.
\newblock {\em Invent. Math. 155}, 1 (2004), 81--161.

\bibitem{GolseStRaymond1}
{\sc Golse, F., and Saint-Raymond, L.}
\newblock Hydrodynamic limits for the {B}oltzmann equation.
\newblock {\em Riv. Mat. Univ. Parma 4}, 7 (2005), 1--144.

\bibitem{haff:1983}
{\sc Haff, P.}
\newblock {Grain flow as a fluid-mechanical phenomenon}.
\newblock {\em J. Fluid Mech. 134\/} (1983), 401--30.

\bibitem{HuangWang:2001}
{\sc Huang, F., and Wang, Z.}
\newblock {Well Posedness for Pressureless Flow}.
\newblock {\em Communications in Mathematical Physics 222}, 1 (Aug. 2001),
  117--146.

\bibitem{KangVasseur:2014}
{\sc Kang, M.-J., and Vasseur, A.}
\newblock {Asymptotic Analysis of Vlasov-type Equation Under Strong Local
  Alignment Regime}.
\newblock preprint arXiv 1412.3119.

\bibitem{LionsPerTad}
{\sc Lions, P., Perthame, B., and Tadmor, E.}
\newblock A kinetic formulation of multidimensional scalar conservation laws
  and related questions.
\newblock {\em J. Amer. Math. Soc. 7\/} (1994), 169--191.

\bibitem{LionsPerTad2}
{\sc Lions, P., Perthame, B., and Tadmor, E.}
\newblock Kinetic formulation of the isentropic gas dynamics and $p$-systems.
\newblock {\em Comm. Math. Phys. 163\/} (1994), 415--431.

\bibitem{MischlerMouhot:20062}
{\sc Mischler, S., and Mouhot, C.}
\newblock {Cooling process for inelastic Boltzmann equations for hard spheres,
  Part II: Self-similar solutions and tail behavior}.
\newblock {\em J. Statist. Phys. 124}, 2 (2006), 703--746.

\bibitem{Oleinik}
{\sc Oleinik, O.}
\newblock On {C}auchy’s problem for nonlinear equations in a class of
  discontinuous functions.
\newblock {\em Doklady Akad. Nauk SSSR (N.S.) 95\/} (1954), 451--454.

\bibitem{Perthame}
{\sc Perthame, B.}
\newblock {\em Kinetic Formulations of Conservation Laws}.
\newblock Oxford series in mathematics and its applications. Oxford University
  Press, 2002.

\bibitem{rey:2011}
{\sc Rey, T.}
\newblock {Blow Up Analysis for Anomalous Granular Gases}.
\newblock {\em SIAM J. Math. Anal. 44}, 3 (2012), 1544--1561.

\bibitem{StRaymond}
{\sc Saint-Raymond, L.}
\newblock {From the Boltzmann BGK equation to the Navier-Stokes system}.
\newblock {\em Ann. Sci. Ecole Norm. Sup. 36}, 2 (2003), 271--317.

\bibitem{SilkSzalayZeldovich:1983}
{\sc Silk, J., Szalay, A., and Zeldovich, Y.~B.}
\newblock {Large-scale structure of the universe}.
\newblock {\em Scientific American 249\/} (1983), 72--80.

\bibitem{Toscani:2008}
{\sc Toscani, G.}
\newblock {\em Mathematical Models of Granular Matter}.
\newblock Springer Berlin Heidelberg, Berlin, Heidelberg, 2008,
  ch.~Hydrodynamics from the Dissipative Boltzmann Equation, pp.~59--75.

\bibitem{Tristani:2013}
{\sc Tristani, I.}
\newblock {Boltzmann Equation for Granular Media with Thermal Forces in a
  Weakly Inhomogeneous Setting}.
\newblock {\em J, Funct. Anal.\/} (2015).
\newblock In Press.

\bibitem{Villani:2006}
{\sc Villani, C.}
\newblock {Mathematics of Granular Materials}.
\newblock {\em J. Statist. Phys. 124}, 2 (2006), 781--822.

\bibitem{Wu:2009}
{\sc Wu, Z.}
\newblock {$L^1$ and BV-type stability of the inelastic Boltzmann equation near
  vacuum}.
\newblock {\em Continuum Mechanics and Thermodynamics 22}, 3 (Nov. 2009),
  239--249.

\end{thebibliography}
  
\end{document}